\documentclass[11pt,twoside]{article}

\author{Thomas K\"{o}ppe}
\title{Computations of instanton invariants}
\date{}

\usepackage[pdftitle={Computations of instanton invariants}, pdfauthor={Thomas Köppe},
            pdfsubject={Mathematics, Algebraic Geometry, Computational geometry}, pdfproducer={}, pdfcreator={},
            pdfpagelayout=TwoColumnRight, colorlinks=true, linkcolor=black, urlcolor=black, citecolor=black]{hyperref}
\usepackage{amsmath,amssymb,amsthm,fancyhdr,booktabs,longtable}
\usepackage[outer=3cm,inner=3cm]{geometry}

\numberwithin{equation}{section}
\newtheorem{proposition}{Proposition}[section]

\newtheorem*{thm.ff}{Theorem on Formal Functions}
\theoremstyle{definition}
\newtheorem{definition}[proposition]{Definition}
\newtheorem{example}[proposition]{Example}
\newtheorem{remark}[proposition]{Remark}
\theoremstyle{remark}

\DeclareMathOperator{\Ext}{Ext}
\DeclareMathOperator{\Hom}{Hom}
\DeclareMathOperator{\SEnd}{\mathcal{E}\!\mathit{nd}}
\DeclareMathOperator{\coker}{coker}
\DeclareMathOperator{\ev}{ev}
\DeclareMathOperator{\Tot}{Tot}
\newcommand{\op}{\mathcal{O}_{\mathbb{P}^1}}
\newcommand{\minU}{\mathrm{min}_u}
\newcommand{\maxU}{\mathrm{max}_u}
\newcommand{\minZ}{\mathrm{min}_z}
\newcommand{\maxZ}{\mathrm{max}_z}
\newcommand{\ce}{\mathrel{\mathop:}=}
\newcommand{\ec}{=\mathrel{\mathop:}}
\newcommand{\eqand}[1][and]{\qquad\text{#1}\qquad}
\newcommand{\mypm}[4]{\begin{pmatrix}#1\\#2\\#3\\#4\end{pmatrix}}
\newcommand{\mysm}[2]{\bigl(\begin{smallmatrix}#1\\#2\end{smallmatrix}\bigr)}
\newcommand{\oton}[2][n]{#2_1, \dotsc, #2_{#1}}
\newcommand{\vfix}{\vphantom{stfj\minU\beta}}

\begin{document}

\maketitle

\begin{abstract}\noindent
Motivated by newly discovered properties of instantons on non-compact
spaces, we realised that certain analytic invariants of vector bundles
detect fine geometric properties. We present numerical algorithms,
implemented in \emph{Macaulay~2}, to compute these invariants.

Precisely, we obtain the direct image and first derived functor of the
contraction map $\pi \colon Z \to X$, where $Z$ is the total space of
a negative bundle over $\mathbb{P}^1$ and $\pi$ contracts the zero
section. We obtain two numerical invariants of a rank-$2$ vector
bundle $E$ on $Z$, the width $h^0\bigl(X; \; (\pi_*E)^{\vee \vee}
\bigl/ \pi_*E\bigr)$ and the height $h^0\bigl(X; \; R^1 \pi_*E
\bigr)$, whose sum is the local holomorphic Euler characteristic
$\chi^\text{loc}(E)$.
\end{abstract}

\tableofcontents

\section{Introduction}

In this paper we present effective algorithms, implemented in
\emph{Macaulay~2}, for the computation of two numerical invariants of
locally free sheaves of rank $2$ and with $c_1=0$ on a family of open
complex surfaces $Z_k$ which contain a distinguished line $\ell$ of
self-intersection $\ell^2=-k$, $k>0$. The interest in these sheaves
arises from mathematical physics, since the Kobayashi-Hitchin
correspondence identifies a certain subset of these sheaves with
instantons on $Z_k$, and in this picture our two numerical invariants
add up to the \emph{local charge} of the instanton near the line
$\ell$. However, the invariants are strictly finer than the charge,
and they apply to a larger class of sheaves than just those which
correspond to instantons, and they provide a way to stratify the
moduli of $\mathfrak{sl}_2$-bundles on $Z_k$ into ``nice'' components.

The mathematical theory behind these sheaves and their relation to
physics has been studied in \cite{gkm}, and the study of their moduli
is the subject of \cite{bgk1} and \cite{bgk2}. The explicit
computation of the numerical invariants has been an essential
ingredient of several of the results in those papers, for the proof of
which one used ``direct computation''. It is the aim of this paper to
describe general algorithms for these direct computations. A reference
implementation in \emph{Macaulay~2} can be found on the author's
website at \url{http://www.maths.ed.ac.uk/~s0571100/Instanton/}.

Finally, we will see that the algorithms are actually easily adaptable
to a larger class of computations of sheaf cohomology on more general
spaces. One such adapted algorithm will be used on our upcoming paper
\cite{gk1}, where we study sheaves on local Calabi-Yau threefolds,
such as $\Tot\bigl( \op(-1) \oplus \op(-1) \bigr)$ and $\Tot\bigl( \op
\oplus \op(-2) \bigr)$. Sheaves on Calabi-Yau threefolds are of
interest to numerous mathematicians and physicists, in particular the
aforementioned spaces appear in the context of brane theory in papers
by Dijkgraaf--Vafa and G.\ Moore.

\paragraph{Outline.} In \S\ref{sec.defs} we give the basic
definitions of the types of sheaves in whose computations we are
interested in this paper, along with some background results. We
define our two basic invariants, the \emph{width} and the
\emph{height}. The algorithms for the explicit computation of the
width and the height are described in \S\ref{sec.algwidth} and
\S\ref{sec.algheight}, respectively. Finally, we describe how to
compute similar invariants of endomorphism bundles in
\S\ref{sec.algend1} and \S\ref{sec.algend0} and how to adapt the
algorithms to other situations.

\paragraph{Acknowledgments.} The algorithms in this paper are
generalisations of the work of I.\ Swanson and E.\ Gasparim in
\cite{gs}. The author is very grateful to D.\ Grayson and M.\ Stillman
for their great work on creating \emph{Macaulay~2} and maintaining a
friendly and active user community, and to E.\ Gasparim for the
privilege of enjoyable collaboration and sage advise.

\section{Definitions and background}\label{sec.defs}

Let $Z_k$ be the total space of the line bundle
$\mathcal{O}_{\mathbb{P}^1}(-k)$ over $\mathbb{P}^1$ and $k >
0$. Denote by $\ell$ the zero section, so that $\ell^2=-k$. Let $E$ be
a holomorphic rank-$2$ vector bundle over $Z_k$ with $c_1(E)=0$. It is
known from \cite{ga2} that $E$ is an algebraic extension of algebraic
line bundles,
\begin{equation}\label{eq.Eext}
 0 \longrightarrow \mathcal{O}(-j) \longrightarrow E
   \longrightarrow \mathcal{O}(j) \longrightarrow 0 \text{ ,}
\end{equation}
where $\mathcal{O}(j)$ is the pull-back to $Z_k$ of
$\mathcal{O}_{\mathbb{P}^1}(j)$ under the projection $Z_k \to
\mathbb{P}^1$.

We fix once and for all coordinate charts $U=\{z,u\}$ and $V=\{w,v\}$
on $Z_k$, so that $w=z^{-1}$ and $v=z^ku$. (A more symmetric picture
is given by weighted homogeneous coordinates $[x_0:x_1:x_2]$ of
degrees $(1,-k,1)$, but this is less useful for explicit
computations.)

In these coordinates, an extension of the form \eqref{eq.Eext} is
given by the transition function from the $U$-chart to the $V$-chart that takes the form of the matrix
\begin{equation}\label{eq.transf}
  T = \begin{pmatrix} z^j & p(z,u) \\ 0 & z^{-j} \end{pmatrix} \text{ .}
\end{equation}
The integer $j \geq 0$ is called the \emph{splitting type} of $E$, and
the function $p$, which is holomorphic on $U \cap V$, represents an
element of $\Ext_{\mathcal{O}_{Z_k}}^1\bigl(\mathcal{O}(-j),
\mathcal{O}(j)\bigr)$. We write $[p]$ for the equivalence class of all
functions that determine isomorphic extensions, and we say that the
pair $(j,p)$ determines the bundle $E$ on $Z_k$. Hence $(j,q)$ defines
isomorphic bundles for all $q \in [p]$.

It turns out that in our case the function $p$ is always a polynomial:
\begin{proposition}[{\cite{bgk1}}]\label{prop.canp}
If $(j,p)$ determines a bundle of type \eqref{eq.Eext} on $Z_k$, then
we can choose $p$ to be a polynomial in $u$, $z$ and
$z^{-1}$. Moreover, the polynomial can be chosen to have the form
\begin{equation}\label{eq.pnormal}
  p(z,u) = \sum_{r=1}^{\bigl\lfloor\frac{2j-2}{k}\bigr\rfloor} \ \sum_{s=kr-j+1}^{j-1} p_{rs} u^r z^s \text{ ,}
\end{equation}
and if $u\bigl|p(z,u)$, the equivalence is given by simply setting all
terms that do not appear in \eqref{eq.pnormal} to zero.\qed
\end{proposition}
\begin{remark}
If $u$ does not divide $p$, then the extension \eqref{eq.transf}
defines a bundle $E$ that is also an extension of line bundles of a
lower splitting type. In that case the problem $(j,p)$ is
ill-posed. This does not happen if $u$ divides $p$.
\end{remark}

We want to compute explicitly from this data the so-called \emph{local
holomorphic Euler characteristic}
\[ \chi^\text{loc}(E) \ce h^0\bigl(X_k; (\pi_*E)^{\vee\vee}\bigl/\pi_*E\bigr)
   + h^0\bigl(X_k; R^1\pi_*E\bigr) \text{ ,} \] where $X_k$ is
obtained from $Z_k$ by contracting the zero-section via $\pi\colon Z_k
\to X_k$. Since $\pi\rvert_{Z_k\backslash\ell}$ is an isomorphism onto
$X_k\backslash\{0\}$, the sheaves $(\pi_*E)^{\vee\vee}\bigl/\pi_*E$
and $R^1\pi_*E$ are supported over the single point $0 \in X_k$, and
so their spaces of global sections are simply their values at $0$. In
symbols, we have
\begin{eqnarray*}
  H^0\bigl(X_k; (\pi_*E)^{\vee\vee}\bigl/\pi_*E\bigr) &=&
  \bigl((\pi_*E)^{\vee\vee}\bigl/\pi_*E\bigr)_0 \ec Q_0 \text{ , and}\\
  H^0\bigl(R^1\pi_*E\bigr) &=& \bigl(R^1\pi_*E\bigr)_0 \text{ .}
\end{eqnarray*}
To compute these stalks on $X_k$ we make heavy use of the Theorem on
Formal Functions and instead compute sections of $E$ on $Z_k$.
\begin{thm.ff}[Grauert, Grothendieck]
Let $\pi \colon Z \to X$ be a proper map of complex spaces and
$\mathcal{F}$ a coherent sheaf on $Z$. For $x \in X$ let $\ell \ce
\pi^{-1}(x)$. Then
\[ \bigl(R^i\pi_*\mathcal{F}\bigr)^\wedge_x \cong
   \varprojlim_n H^i\bigl(\ell^{(n)}; \mathcal{F}
   \bigr\rvert_{\ell^{(n)}}\bigr) \text{ ,} \]
where $\ell^{(n)}$ denotes the $n^\text{th}$ infinitesimal
neighbourhood of $\ell$ in $Z$.
\end{thm.ff}

\section{Computation of \texorpdfstring{$h^0\bigl(X_k; (\pi_* E)^{\vee\vee}
\bigl/ \pi_* E \bigr)$}{h\^{}0(pi\_*E**/pi\_*E)}}\label{sec.algwidth}

Let $E$ be a vector bundle on $Z_k$ determined by the data $(j,p)$ as
described in the previous section. We want to compute the dimension of
the vector space $Q_0 \ce \bigl((\pi_*E)^{\vee\vee} \bigl/
\pi_*E\bigr)_0$, which is the stalk at $0$ of the skyscraper sheaf $Q$
defined by the exact sequence
\[ 0 \longrightarrow \pi_* E \xrightarrow{\ \ev\ } (\pi_* E)^{\vee\vee}
   \longrightarrow Q \longrightarrow 0 \text{ .} \]
By definition, the $\mathcal{O}_{X_k}$-module structure on $\pi_*E$ is
determined by the lifting map
\begin{equation}\label{eq.lift}
  \widetilde{\pi} \colon \mathcal{O}_{X_k} \to \pi_* \mathcal{O}_{Z_k} \text{ ,}
\end{equation}
which is an isomorphism away from $0$ and whose stalk at $0$ is, in
$U$-coordinates, just $\widetilde{\pi}(w_i)=z^iu$, where
\[ \mathcal{O}_{X_k,0} \ec S = \mathbb{C}\bigl[w_0, w_1, \dotsc,
   w_k\bigr] \bigl/ \bigl(w_i w_j - w_{i+1} w_{j-1}\bigr) \text{ ,} \]
where the ideal contains all the indices $i=0, \dotsc, k-2$ and
$j=i+2,\dotsc,k$.

Thus we are lead to compute the space of sections of $E$, first as
$\mathbb{C}[z,u]$-module and then as an $S$-module. Since we are using
the Theorem on Formal Functions for the computation, we will actually
be computing the $\mathbb{C}[[z,u]]$- and $S^\wedge$-module
structures. However, on any Noetherian locally ringed space
$\bigl(X,\mathcal{O}\bigr)$ the completion $\widehat{\mathcal{O}}$ is
a flat $\mathcal{O}$-module, and $\mathcal{F}_x^\wedge = \mathcal{F}_x
\otimes_\mathcal{O} \widehat{\mathcal{O}}$ for any $\mathcal{O}$-sheaf
$\mathcal{F}$ and all $x \in X$.

The computation of $\pi_*E$ proceeds in three steps: First we apply
the Theorem on Formal Functions to express $(\pi_*E)_0$ in terms of
the cohomology of $E$, i.e.\
\[ \bigl(\pi_*E\bigr)_0^\wedge = \varprojlim_n H^0\bigl(\ell^{(n)};
   \; E^{(n)} \bigr) \text{ .} \]
In local coordinates, elements of $H^0\bigl(\ell^{(n)};\;E^{(n)}
\bigr)$ are sections of $E$ of the form $\sigma = \bigl(a(z,u),
b(z,u)\bigr)$, where $a$ and $b$ are \textit{a priori} power series in
\[ \mathcal{O}_{\ell^{(n)}}(U) = \bigoplus_{i=0}^n u^i .
   \mathbb{C}\bigl\{z\,\bigr\} \text{ .} \]
However, we require that the local section patch correctly onto the
other chart, so that $T\sigma = \bigl(z^j a + pb, z^{-j}b\bigr)$ is a
holomorphic section of $E^{(n)}(V)$, i.e.\ holomorphic in
$\bigl(z^{-1}, z^ku\bigr)$. This shows that $a$ and $b$ are in fact
polynomials, i.e.\ they contain only finitely many non-zero powers of
$z$.

\begin{remark}
We have essentially demonstrated the GAGA correspondence for the
projective schemes $\ell^{(n)}$: all holomorphic sheaves are
algebraic; and even if we start with a holomorphic section of a sheaf,
we are forced to conclude that it is algebraic. Consequently, it is
irrelevant in our computations of cohomology whether we consider
$\ell^{(n)}$ as an algebraic scheme (with the Zariski topology) or a
complex analytic space (with the Euclidean topology).
\end{remark}

For the second step, we have to show that we can compute the module
structure of $(\pi_*E))_0^\wedge$ from a finite amount of data
(essentially by only going up to a finite infinitesimal neighboorhood,
but see \S\ref{sec.sss_bounds}). To be slightly more precise, we
will \emph{not} compute $H^0\bigl(\ell^{(n)}; \; E^{(n)} \bigr)$, but
instead we will identify finitely many elements in $H^0\bigl(\ell^{(n)};
\; E^{(n)} \bigr)$ that generate $(\pi_*E))_0^\wedge$ as a
$\mathbb{C}[[z,u]]$-module. (The fact that we can do this depends
crucially on the structure of the space $Z_k$ and the fact that the
conormal bundle of $\ell \subset Z_k$ is ample.)




Finally, once we have computed
\[ M \ce H^0\bigl(\ell^{(N)};\;E^{(N)} \bigr) \]
(for some sufficiently large $N$) as a $\mathbb{C}[[z, u]]$-module,
the third and final step is to find the $S^\wedge$-module structure on
$M$ induced by the lifting map \eqref{eq.lift}. Here we exploit the
fact that $u$ is not a zero-divisor in $\mathbb{C}[[z,u]]$ and that
every element in $\mathbb{C}[[z,u]]$ can be expressed in terms of
$w_i=z^iu$ after multiplication by a sufficiently high power of $u$.

\subsection{An example}\label{sec.example}

An example is often more illuminating than a detailed theoretic
description of an algorithm, so let us start with a typical one:

\begin{example}\label{ex.mainexample}
Consider $Z_2$, the total space of $\mathcal{O}_{\mathbb{P}^1}(-2)$,
whose blow-down $X_2$ is a surface with an ordinary double point,
which for convenience we give coordinates $x=w_0=u$, $y=w_1=zu$ and
$w=w_2=z^2u$, where $xw=y^2$. Let the $E$ be the bundle on $Z_2$
determined by the extension $p(z,u)=u$ and of splitting type $j=3$. To
compute the module $M$ we are thus looking for sections $(a,b)$ of $E$
such that
\begin{equation}\label{eq.ex1}
  \begin{pmatrix} z^3 a + u b \\ z^{-3} b \end{pmatrix}
\end{equation}
is holomorphic in $z^{-1}$ and $z^2u$. Since $a$ and $b$ are
holomorphic in $(z,u)$, we can write
\begin{align*}
  a(z,u) &= \sum_{r,s\geq0} a_{rs} u^r z^s & \text{and} &&
  b(z,u) &= \sum_{r,s\geq0} b_{rs} u^r z^s \text{ .}
\end{align*}

The basic idea is to work ``one infinitesimal neighbourhood at a
time'', i.e.\ to deal with each power of $u$ separately, starting from
$0$, until one has ``enough'' information. Thus, starting at $r=0$,
we see from the second entry of \eqref{eq.ex1} that $z^{-3} b(z,u) \pmod{u}$
has to be holomorphic in $z^{-1}$, so  that
\[ b(z,u) \equiv b_{00} + b_{01}z + b_{02}z^2 + b_{03}z^3 \pmod{u} \text{ .} \]
Next, there can be no holomorphic terms in $a$ with $r=0$. Thus one
continues at $r=1$:
\[ a(z,u) \equiv (a_{10}u + a_{11}uz + a_{12}uz^2 + \dotsb) \pmod{u^2} \text{ .} \]
The term $z^3 a_{10}u$ is not holomorphic in
$\bigl(z^{-1},z^2u\bigr)$, but it is matched by the subsequent $u
b_{03}z^3$. Thus we have a relation: $a_{10} + b_{03} = 0$. There are
no further terms in $a(z,u)$ for $r=1$ that can be matched by
$u\,b(z,u)$, so $a_{1s}=0$ for $s \geq 1$.

We could now carry on to the next formal neighbourhood, find more terms for $b$
\[ b_{10} u + b_{11} uz + b_{12} uz^2 + b_{13} uz^3 + b_{14} uz^4 + b_{15} uz^5 \text{ ,} \]
and then calculate relations on $a_{2s}$. However, we shall see
immediately that this adds no new information to the $S$-module of
sections.

We must now find generators of the sections of $E$, but considered as
a module \emph{over $S$}, where $S$ is the ring
\[ S = \mathbb{C}[x,y,w]\bigl/\bigl\langle xw-y^2 \bigr\rangle \text{ .} \]
We certainly have the following generators:
\[ \beta_0=\begin{pmatrix}0\\1\end{pmatrix}\;,\quad
   \beta_1=\begin{pmatrix}0\\z\end{pmatrix},\;\quad\text{and}\quad
   \beta_2=\begin{pmatrix}0\\z^2\end{pmatrix} \text{ .} \]
But also, we must take into account the relation $a_{10} + b_{03} =
0$. Thus the last generator is
\[ \gamma = \begin{pmatrix}-u\\z^3\end{pmatrix} \text{ .} \]
Of course the module is not \emph{free} over $S$, since we have the
following relations:
\begin{align*}
  y\beta_0 &= x\beta_1 & y\beta_1 &= x\beta_2 \\
  w\beta_0 &= y\beta_1  & w\beta_1 &= y\beta_2
\end{align*}
The computation is actually complete now: Even if one were to consider
higher generators, e.g.\ the section $\beta_4 = (0,u)$ coming from the
$b_{10}$-term, it would just be a multiple of an existing generator
(here $\beta_4 = x\beta_0$). Also, the $a_{20}$-term appears to
provide a new, free generator $\alpha = (u^2,0)$, since $z^3 a_{20}
u^2$ is actually holomorphic on $V$; however, we have $\alpha =
y\beta_2 - x\gamma$.

At this stage of our guiding example we record the $S$-module of
sections that we just computed:
\[ M \ce S\bigl[\beta_0, \beta_1, \beta_2, \gamma\bigr] \bigl/ \bigl( y\beta_0 - x\beta_1 ,
   w\beta_0 - y\beta_1 , y\beta_1 - x\beta_2 , w\beta_1 - y\beta_2 \bigr) \]

We now proceed to compute $M^\vee$, $M^{\vee\vee}$ and the quotient.
The simplest, linearly independent elements of $M^\vee$ that we can
write down are
\[ \beta^\vee = \bigl\{ \beta_0 \mapsto x, \beta_1 \mapsto y, \beta_2 \mapsto w,
   \gamma \mapsto 0 \bigr\} \qquad\text{and}\qquad \gamma^\vee = \bigl\{ \beta_i
   \mapsto 0, \gamma \mapsto 1 \bigr\} \text{ .} \]
A moment's reflection shows that all other possible maps are
combinations or $S$-multiples of these two generators, and clearly
there are no relations. Thus $M^\vee$ is the free $S$-module
\[ M^\vee = S\bigl[\beta^\vee, \gamma^\vee \bigr] \text{ .} \]
The bi-dual is now simply
\[ M^{\vee\vee} = S\Bigl[\beta^{\vee\vee} = \bigl\{ \beta^\vee \mapsto 1 \bigr\},
   \gamma^{\vee\vee} = \bigl\{ \gamma^\vee \mapsto 1 \bigr\} \Bigr] \text{ .} \]
Evaluation on $M$ yields:
\begin{align*}
  \ev(\beta_0) &= \begin{Bmatrix}\beta^\vee \mapsto x \\ \gamma^\vee \mapsto 0 \end{Bmatrix} = x\beta^{\vee\vee} &
  \ev(\beta_1) &= \begin{Bmatrix}\beta^\vee \mapsto y \\ \gamma^\vee \mapsto 0 \end{Bmatrix} = y\beta^{\vee\vee} \\
  \ev(\beta_2) &= \begin{Bmatrix}\beta^\vee \mapsto w \\ \gamma^\vee \mapsto 0 \end{Bmatrix} = w\beta^{\vee\vee} &
  \ev(\gamma)  &= \begin{Bmatrix}\beta^\vee \mapsto 0 \\ \gamma^\vee \mapsto 1 \end{Bmatrix} = \gamma^{\vee\vee}
\end{align*}
So
\[ \coker\bigl(\ev\bigr) = \bigl\langle \beta^{\vee\vee} \bigr\rangle_{\mathbb{C}} \text{ ,} \]
which has dimension $1$. So $l(Q)=1$. \hfil\qed
\end{example}

\subsection{Description of the algorithm}

The example of \S\ref{sec.example} suggests a general algorithm:
We must consider two \emph{polynomials} $a$ and $b$, use the condition
that $z^j a(z,u) + p(z,u)b(z,u)$ and $z^{-j}b(z,u)$ be holomorphic in
$\bigl(z^{-1}, z^ku \bigr)$ to obtain relations on the coefficients
$a_{rs}$ and $b_{rs}$, thence create the $S$-module $M$, and finally
compute the dimension of the quotient $M^{\vee\vee}\bigl/M$.

The crucial consideration is that we only need consider \emph{finitely
many} terms in $a(z,u)$ and $b(z,u)$, and this will suffice to
describe the module structure of $M$. In other words, we guarantee
that we can choose \textit{a priori} polynomials
\[ a(z,u) = \sum_{r=0}^{A_1} \sum_{s=0}^{A_2} a_{rs}z^su^r \quad\text{and}\quad
   b(z,u) = \sum_{r=0}^{B_1} \sum_{s=0}^{B_2} b_{rs}z^su^r \text{ ,} \]
in which we treat the coefficients $a_{rs}$ and $b_{rs}$ as
indeterminates, which toghether with the finitely many relations among
them generate the module $M$. The bounds $A_1$, $A_2$, $B_1$ and $B_2$
will only depend on $k$, $j$ and $p$, and they will be determined at
the start of the algorithm. This is described in \S\ref{sec.sss_bounds}.

Following the computation of the relations among the coefficients, we
require a small, technical routine to convert the $\mathbb{C}[z,u]$-module
into an $S$-module. These technical algorithms are described at the end of
\S\ref{sec.ss_imp}. Finally, for the computation of the quotient
$M^{\vee\vee}\bigl/ M$ we use the same computational method that was
described in \cite[Lemma 2.1 (iii)]{gs}.

\subsubsection{Computation of \texorpdfstring{$M$}{M}
}\label{sec.sss_bounds}

Intuitively, it is clear that to compute $M$ one has to write down
``enough'' terms of $a$ and $b$, calculate $f \ce z^j a + p b$ and set
to zero all terms in $f$ that are not holomorphic in $z^{-1}$ and
$z^ku$. This gives a set of relations among the coefficients $a_{rs}$
and $b_{rs}$, which in turn determines a set of sections that generate
$M$. (In the example of \S\ref{sec.example}, the relation
$a_{10}+b_{03}=0$ implied the generator $\gamma=(-u,z^3)^T$.) In this
section we give precise instructions on how to find the relations
among coefficients and how to build from them a generating set of
sections.

First let us fix some notation: To each coefficient $a_{rs}$ and
$b_{rs}$, let us associate, respectively, ``elementary'' sections
\[ \sigma\bigl(a_{rs}\bigr) \ce \sigma_{rs}^a \ce \begin{pmatrix}z^s u^r \\ 0\end{pmatrix} \eqand
   \sigma\bigl(b_{rs}\bigr) \ce \sigma_{rs}^b \ce \begin{pmatrix}0 \\
z^s u^r\end{pmatrix} \text{ .} \]
Then the generator associated to a relation $R = a_{rs} + \sum_{il}
R_{il} b_{il} = 0$, where $R_{il}$ is non-zero for at least one
$(i,l)$, is $\sigma(R) \ce -\sigma_{rs}^a + \sum_{il} R_{il}
\sigma^b_{il}$. We denote by $\mathcal{R}$ the set of all such
relations, so we may consider $\mathcal{R}$ to be the ``solution set''
of the holomorphy condition $T\mysm{a}{b}\bigl/\Gamma(E;V) = 0$. With
this notation, $M$ is generated as a $\mathbb{C}[z,u]$-module by the
set $G_\mathcal{R} \ce \bigl\{ \sigma(R) : R \in \mathcal{R}\bigr\}$, and as an
$S$-module by $G'_\mathcal{R} \ce \bigl\{ \pi_* \sigma(R) : R \in \mathcal{R}\bigr\}$.

There are two problems one faces when restricting oneself to a
(finite) polynomial, which we turn into
\paragraph{Objectives for the algorithm.}
\begin{enumerate}
\item One must find all generators of $M$, i.e.\ one must ensure that
      $G_\mathcal{R}$ generates $M$. For example, on $Z_2$ with $p=0$
      and $j=4$, the $a_{20}$-term contributes a free generator
      $(u^2,0)$, which one could miss by only considering the $r=0$
      and $r=1$ infinitesimal neighbourhoods for $a$.
\item One must find all relations between $b_{r's'}$- and
      $a_{rs}$-terms. Some $b_{rs}$-terms may appear to be free when
      one does not consider enough $a_{rs}$-terms. For example, on
      $Z_2$ with $j=5$ and $p=u^2$, the term $b_{05}z^5$ may
      erroneously seem to constitute the free generator $(0,z^5)$ if
      one does not include the second infinitesimal neighbourhood
      and finds $a_{20}+b_{05} = 0$, so that the actual generator
      is $(-u^2, z^5)$.
\end{enumerate}
There exists a precise bound on the number of infinitesimal
neighbourhoods which one needs to consider. By including terms from a
higher neighbourhoods into the polynomials $a$ or $b$, one may see new
relations involving terms from lower neighbourhoods appear, but at the
same time this will add new generating terms for which one might in
turn be tempted to find new relations in even higher
neighbourhoods. However, we have \textit{a priori} bounds on the terms
in $a$ and $b$ that ensure that we compute the correct module
structure on $M$.
\begin{enumerate}\setcounter{enumi}{2}
\item It is acceptable for $\mathcal{R}$ to contain too many relations
      involving terms in $a$. This happens when there are not enough
      terms in $b$ to match. In \cite{gs} this was called a ``fake
      relation''. However, if $R\in\mathcal{R}$ is such a fake
      relation, and if by considering higher terms we would find the
      corresponding ``real'' relation to be $R'$, then we can ensure
      that $\sigma(R')$ is already contained in the module generated
      by $G_\mathcal{R}$.

      This will inevitably be the case when $p$ contains several terms
      of different degree in $u$: In that case one cannot possibly
      find all correct relations among a finite set of terms. The key
      is to allow high terms of $a$ to be set to zero ``erroneously'',
      rather than to miss a relation between a term $b_{r's'}$ and a
      term $a_{rs}$. (The latter would cause us to add
      a wrong generator, while the former only removes a potential
      generator -- but we are careful to miss only multiples of
      earlier generators.)

      We illustrate this important point by means of Example
      \ref{ex.mainexample}: Suppose we only considered $b$ up to the
      neighbourhood $r=0$, and $a$ up to $r=2$. Then we had to
      conclude the relation $R:a_{22}=0$, which is a ``fake
      relation'', whose corresponding ``real'' relation is
      $R':a_{22}+b_{15}=0$. However, $\sigma(R') = uz^2\sigma(\gamma)
      = w_2\sigma(\gamma)$, so we do not need the generator
      $\sigma(R')$.
\end{enumerate}

\noindent The range of coefficients which one needs to consider
depends on the extension $p$:
\begin{definition} Let $p \in \mathbb{C}[z,z^{-1},u]$. We define:
\begin{itemize}
\item $\minU \ce $ the minimal degree of $u$ occurring in $p$,
\item $\maxU \ce $ the maximal degree of $u$ occurring in $p$,
\item $\minZ \ce $ the minimal degree of $z$ occurring in $p$, and
\item $\maxZ \ce $ the maximal degree of $z$ occurring in $p$.
\item If $p \equiv 0$, then all the above values would be $-\infty$;
however, for this case we define $\minU \ce 0$, which will later save
us from having to consider this case separately.
\end{itemize}
\end{definition}

For a given bundle $E$ on $Z_k$ determined by $(j,p)$, there are
immediate bounds on the number of degrees of $z$ that need to be
considered for each fixed $r$:

\begin{proposition}\label{prop.bsbounds}
For any degree $r$ of $u$ and independent of $p$, only the terms
\[ b_{r0} u^r + \dotsb + b_{r,kr+j} u^r z^{kr+j} \]
occur in $b$.
\end{proposition}
\begin{proof}
We require that $z^{-j}b(z,u)$ be holomorphic in $z^{-1}$ and
$z^ku$. By multiplying $\sum_{s=0}^\infty b_{rs}z^s u^r$ by $z^{-j}$
we see that the only terms that are holomorphic in $z^{-1}$ and $z^ku$
are those claimed.
\end{proof}

\begin{proposition}\label{prop.asbounds}
For any $r < \minU$, the terms in $a$ of degree $r$ in $u$, if any, are
\[ a_{r0} u^r + \dotsb + a_{r,kr-j} u^r z^{kr-j} \text{ .} \]
\end{proposition}
\begin{proof}
Since $r < \minU$, no term $a_{rs} u^r z^s$ can be combined with any
term in $pb$, so the problem reduces to making $z^j a_{rs}u^r z^s$
holomorphic in $z^ku$, which results precisely in those terms stated.
\end{proof}

\begin{proposition}\label{prop.as2bounds}
For $r \geq \minU$, the only terms $a_{rs} u^r z^s$ that can possibly
be non-zero satisfy
\[ 0 \leq s \leq \max\bigl\{ k(r-\minU) + j + \max\{\maxZ,0\},
   \ kr-j\bigr\} \text{ .} \]
\end{proposition}
\begin{proof}
Consider all terms in $z^j a_{rs} u^r z^s $ that are not holomorphic
in $z^{-1}$ and $z^k u$: They must vanish unless they can be matched
by a term in $pb$. The only terms in $pb$ that have degree $r$ in $u$
are of the form $b_{r^\prime s^\prime} u^{r^\prime} z^{s^\prime}$,
where $r-\maxU \leq r^\prime \leq r-\minU$. Since the terms in $b$ are
as in Proposition \ref{prop.bsbounds}, $s$ has to run at least up to
$kr^\prime_\mathrm{max} + j = k(r-\minU) + j$, but the multiplication
$pb$ may have shifted the term matching $a_{rs}$ by up to
$\max(0,\maxZ)$ places up, which explains the first term in the
statement.

Secondly, terms up to $s=kr-j$ are automatically holomorphic in the
expression $z^j a$, so if $kr-j$ is greater than the previous
expression, all terms up to $kr-j$ must be considered, and all the
coefficients are free.
\end{proof}

Finally, we must turn the Objectives 1--3 into ranges for $r$ that we
choose to consider.

\begin{proposition}\label{prop.rsbounds}
By considering only a truncated generic section
\begin{equation}\label{eq.truncsection}
  \begin{pmatrix}a\\b\end{pmatrix} = \sum_{r=0\vfix}^{\minU-1\vfix}
  \ \sum_{s=0\vfix}^{kr-j\vfix} a_{rs} \begin{pmatrix}z^s u^r\\0\end{pmatrix}
  + \sum_{r=\minU\vfix}^{\alpha\vfix} \ \sum_{s=0\vfix}^{\beta\vfix} a_{rs}
  \begin{pmatrix}z^s u^r\\0\end{pmatrix} + \sum_{r=0\vfix}^{\gamma\vfix}
  \ \sum_{s=0\vfix}^{kr+j\vfix} b_{rs}\begin{pmatrix}0\\z^s u^r\end{pmatrix} \text{ ,}
\end{equation}
where
\begin{eqnarray*}
  \alpha &\ce& \max\bigl\{\lceil j/k \rceil, \maxU\bigr\} + \minU \text{ ,} \\
  \beta &\ce& \max\bigl\{kr-j, \ k(r-\minU)+j+\max\{\maxZ,0\} \bigr\} \text{ , and} \\
  \gamma &\ce& \max\bigl\{\lceil j/k \rceil, \maxU\bigr\} \text{ ,}
\end{eqnarray*}
%
%
one finds enough generators to compute the $S$-module $M$.
\end{proposition}
\begin{remark}
This statement contains two facts: First, we claim that our choice of
polynomials $a$ and $b$ gives enough coefficients from which we form
the generators $G_\mathcal{R}$ of $M$. Secondly, we claim that the set
$\mathcal{R}$ of relations is \emph{correct} in the following sense:
If we set $A = B = \infty$ and if $\mathcal{R}_\infty$ denotes the
associated set of relations, then one of two things happens for each
$R \in \mathcal{R}_\infty$: Either $R$ is already a relation in
$\mathcal{R}$, or $\sigma(R)$ is an $S$-multiple of $\sigma(R')$ for
some $R' \in \mathcal{R}$. (This case was illustrated in Objective~3.)
\end{remark}
\begin{proof}
Let us denote the three big sums in \eqref{eq.truncsection} by
$\Sigma_1$, $\Sigma_2$ and $\Sigma_3$ respectively from left to right.
Since the section $(a,b)$ is holomorphic on $U$, we must have
$s\geq0$, and the upper bounds for $s$ in each of the three sums is
given respectively by Propositions \ref{prop.asbounds},
\ref{prop.as2bounds} and \ref{prop.bsbounds}. To justify the choice of
the remaining bounds, consider the condition
\[ T\begin{pmatrix}a\\b\end{pmatrix} = \begin{pmatrix}z^j \, a + p(z,u)\,b
   \\ z^{-j}b\end{pmatrix} \in \Gamma(E;V) \text{ .} \]
The term $\Sigma_1$ is seen to contribute free generators of $M$, since
no term in $T \Sigma_1$ can be matched by any term from $T \Sigma_3$.

The important choice is that of the bound $\alpha$. Once this has been
chosen, we will only consider $b$ up to $u$-degree $\alpha-\minU$,
such that $p(z,u)\,b$ is matched with $a$ (this justifies Objective~3
above). This will justify the choice of $\gamma = \alpha -
\minU$. Moreover this ensures that there cannot be any generators
coming from $a$ that are erroneously considered as free. It remains to
prove that our choice of the bound $\alpha$ leads to correct
computation of the module $M$.

By construction, all the generators we get from $a$ are correct, while
the generators coming from $b$ are either correct or fake. We have to
show two things: (i) all fake relations are multiples of genuine
relations, and (ii) any relation of $M$ is a multiple of a relation
that we have already found. But both (i) and (ii) follow directly from
the choice of $\alpha$.
\end{proof}

\subsubsection{Computation of \texorpdfstring{$M^{\vee\vee}$}{M**} and \texorpdfstring{$l(Q)$}{l(Q)}}

This last section is merely included for completeness. It is no
computational obstacle to compute the dual and bi-dual of $M$:
\[ M^\vee \ce \Hom_S\bigl(M, S\bigr) \qquad\text{and}\qquad
   M^{\vee\vee} \ce \Hom_S\bigl(M^\vee, S\bigr) \text{ .} \]
The \emph{evaluation map} $\ev \colon M \hookrightarrow M^{\vee\vee}$
is the natural map given by
\[ \ev(a) \colon \phi \mapsto \phi(a) \quad \text{for all } a \in M \text{, } \phi \in M^\vee \text{.} \]
Lastly, note that dimension is invariant under completion, i.e.\ $\dim
Q_0^\wedge=\dim Q_0$, so we have $l(Q) \dim\bigl(\coker(\ev)\bigr)$.

\subsection{Implementation of the algorithm}\label{sec.ss_imp}

Our reference implementation of the algorithm is written in
\emph{Macaulay~2} \cite{M2}, a computer algebra package for
commutative algebra. The technical aspects of this implementation are
specific to that language, and the \emph{Macaulay~2}-code is available
from the author's website at \url{http://www.maths.ed.ac.uk/~s0571100/Instanton/}.
Here we present the generic part of the algorithm in pseudo-code.

\paragraph{Input and output.} The algorithm takes as input the data $(k,p,j)$,
where $k > 0$ and $j \geq 0$ are integers and $p$ is a polynomial in
$\bigl(z^{\pm1},u\bigr)$.  The main function \texttt{iWidth(k,p,j)}
computes the width $l(Q)$ for the bundle $E$ on $Z_k$ determined by
$(j,p)$.

\paragraph{Auxiliary functions.} The main function \texttt{iWidth(k,p,j)} calls
several auxiliary functions: The function
\texttt{makeSectionsAndRing(k,p,j)} creates the polynomials $a(z,u)$
and $b(z,u)$ according to Propositions \ref{prop.bsbounds},
\ref{prop.asbounds} and \ref{prop.rsbounds}. The function
\texttt{getRelations(k,fTv)} computes the relations among the
coefficients of $a$ and $b$, where $\mathtt{fTv}=z^ja(z,u) +
p(z,u)b(z,u)$. (Note that \texttt{fTv} contains all the necessary
information.) The function \texttt{makeModule} constructs the
$S$-module $M$ from the data \texttt{aPoly} and \texttt{bPoly}, which
arise respectively from $a(z,u)$ and $b(z,u)$ by applying all the
relations. (For example, if $a_{20}+b_{05}=0$ is a relation, then we
substitute $a_{20}\to b_{05}$ in $a$.) Finally, \texttt{qLength(M)}
computes $l(Q)$ from the module $M$; for its implementation we refer to $\cite{gs}$.

\paragraph{The main function.} Name: \texttt{iWidth}. Input:
\texttt{(k,p,j)}. Output: the instanton width $l(Q)$.\\\emph{Pseudo code.}\\
\begin{raggedright}
\verb!  {aPoly, bPoly, allVars} := makeSectionsAndRing(k, p, j)!\\
\verb!  fTv    := z^j * aPoly + p * bPoly!\\
\verb!  relRes := getRelations(k, fTv)!\\
\qquad  apply substitutions from \texttt{relRes} to \texttt{aPoly} and \texttt{bPoly}\\
\verb!  M      := makeModule(k, aPoly, bPoly, allVars)!\\
\verb!  return qLength(M)!
\end{raggedright}

\paragraph{Auxiliary function.} Name: \texttt{makeSectionsAndRing}. Input:
\texttt{(k,p,j)}. Output: the polynomials $a(z,u)$ and $b(z,u)$, and \texttt{allVars}, a collection of all coefficients occurring in $a(z,u)$ and $b(z,u)$.\\\emph{Pseudo code.}\\
\begin{raggedright}
\verb!  minU := ! minimal $u$-degree of $p$\\
\verb!  maxU := ! maximal $u$-degree of $p$\\
\verb!  minZ := ! minimal $z$-degree of $p$\\
\verb!  maxZ := ! maximal $z$-degree of $p$\\[\baselineskip]
\verb!  aMax    := max(ceiling(j/k), maxU) + minU!\\
\verb!  bMax    := aMax - minU!\\
\verb!  if p = 0 then ( minU = 0; bMax = 0; aMax = ceiling(j/k))!\\[\baselineskip]
\qquad generate coefficients:\\
       \quad $a_{rs}$ such that $r = 0,\dotsc, \mathtt{minU}-1$ and $s = 0,\dotsc,kr-j$;\\
       \quad $a_{rs}$ such that $r = \mathtt{minU}, \dotsc, \mathtt{aMax}$ and $s = 0, \dotsc, \max\bigl\{kr-j, k(r-\mathtt{minU}) + j + \max\{\mathtt{maxZ},0\}\bigr\}$;\\
       \quad $b_{rs}$ such that $r = 0, \dotsc, \mathtt{bMax}$ and $s = 0, \dotsc, kr+j$.\\[\baselineskip]
\verb!  aPoly   := sum(a_(r,s) z^s u^r)!\\
\verb!  bPoly   := sum(b_(r,s) z^s u^r)!\\
\verb!  allVars := !collection of all coefficients \verb!a_(r,s)! and \verb!b_(r,s)!\\[\baselineskip]
\verb!  return { aPoly, bPoly, allVars }!
\end{raggedright}

\paragraph{Main algorithm.} Name: \texttt{getRelations}. Input:
\texttt{(k,fTv)}, where $\mathtt{fTv}=z^ja(z,u) +
p(z,u)b(z,u)$. Output: A collection of relations like
$\bigl\{ a_{20}+b_{05}, \; a_{31}+b_{14}+b_{06}\bigr\}$.\\
\emph{Synopsis.} Each relation is the coefficient of a monomial
$z^s u^r$ in \texttt{fTv} for which $s>kr$.\\
\emph{Pseudo code.}\\
\begin{raggedright}
\verb!  rels   := { }!\\
\verb!  expSet := !set of exponents $(r,s)$ appearing in \texttt{fTv}\\[\baselineskip]
\verb!  for each (r,s) in expSet do!\\
\verb!    if (s <= k * r) then continue!\\
\verb!    term := !the $z^su^r$-term in \texttt{fTv}\\
\verb!    rel  := !the coefficient of \texttt{term}, scaled to be monic\\
\verb!    rels := rels + { rel }!\\
\verb!  end for!\\[\baselineskip]
\verb!  return rels!
\end{raggedright}

\paragraph{Auxiliary function.} Name: \texttt{makeModule}. Input:
\texttt{(k, aPoly, bPoly, allVars)}. Here \texttt{aPoly} and
\texttt{bPoly} are the results of substituting all the relations into
the original $a(z,u)$ and $b(z,u)$. Output: the $S$-module $M$ (e.g.\
its presentation matrix over $S$).\\
\emph{Synopsis.} Iterating over each coefficient in \texttt{allVars},
we set this coefficient to $1$ and all others to $0$ to get a section
$(a,b)$ of $E$. By multiplying with a high power of $u$ (called
\texttt{uexp}), we can express $(u^N a, u^N b)$ as a section of
$\pi_*E$, and those sections generate $M$.\\\emph{Pseudo code.}\\
\begin{raggedright}
\verb!  S       := makeRing(k)!\\
\verb!  Smodule := image! $\mysm00 \colon S^1 \to S^2$\\
\verb!  N       := !the maximum of $s-kr$ over all monomial terms $z^su^r$ in \texttt{aPoly} and \texttt{bPoly}\\
\verb!  uexp    := ceiling(N/k)!\\
\verb!  aPoly   := aPoly * u^uexp!\\
\verb!  bPoly   := bPoly * u^uexp!\\[\baselineskip]
\verb!  for each !coefficient\verb! c in allVars do!\\
\verb!    a := aPoly! with $c=1$ and all other coefficients $=0$\\
\verb!    b := aPoly! with $c=1$ and all other coefficients $=0$\\
\verb!    Smodule := Smodule + image! $\mysm{\mathtt{piStar(a,S)}}{\mathtt{piStar(b,S)}} \colon S^1 \to S^2$\\
\verb!  end do!\\[\baselineskip]
\verb!  return Smodule!
\end{raggedright}\\[.5\baselineskip]
This function calls two further auxiliary functions,
\texttt{makeRing(k)} and \texttt{piStar}. The first one,
\texttt{makeRing(k)}, returns the quotient ring $S \ce
\mathbb{C}[w_0,\dotsc,w_k] \bigl/ (w_i w_j - w_{i+1}w_{j-1})$ for
$i=0, \dotsc, k-2$ and $j=i+2, \dotsc, k$. The second function,
\texttt{piStar}, converts monomials $u^rz^s$ into monomials $\prod_i
w_i^{n_i}$ in $S$, where $\sum_i n_i = r$ and $\sum_i in_i = s$. This
is possible because we multiplied every term by the sufficiently high
power $u^\mathtt{uexp}$.

\paragraph{Auxiliary function.} Name: \texttt{piStar}. Input: \texttt{(p,S)},
where \texttt{p} is some polynomial in $u$ and $z$ in which each term
is of sufficiently high degree in $u$, and \texttt{S} is the target
ring. Output: The polynomial \texttt{p} expressed in
$w_i$-coordinates, where $w_i = z^iu$.\\
\emph{Pseudo code.}\\
\begin{raggedright}
\verb!  res := 0                              //! this will store the result\\
\verb!  k   :=! the number $k$ if \texttt{S} is $\mathbb{C}[w_0,\dotsc,w_k]\bigl/(w_iw_j-w_{i+1}w_{j-1})$ \\
\verb!                                        //! we have variables \texttt{w\_0}, \ldots, \texttt{w\_k}\\
\verb!  for each! term \verb!t in p do!\\
\verb!    degU := ! $u$-degree of \texttt{t}\\
\verb!    degZ := ! $z$-degree of \texttt{t}\\
\verb!    fctr := 1!\\
\verb!    !\\
\verb!    if degZ > k * degU then!\\
\verb!      ! error: this term is not convertible!\\
\verb!    end if!\\
\verb!    !\\
\verb!    diff := k!\\
\verb|    while (diff != 0) do|\\
\verb!      fctr := fctr * w_diff^(degZ/diff)  // degZ/diff! is integer division\\
\verb!      degU := degU - (degZ/diff)!\\
\verb!      degZ := degZ modulo diff!\\
\verb!      diff := diff - 1!\\
\verb!    end do!\\
\verb!    !\\
\verb!    fctr := fctr * w_0^degU!\\
\verb!    res  := res + fctr * (!coefficient of \texttt{t)}\\[\baselineskip]
\verb!    return res!
\end{raggedright}\\[.5\baselineskip]

At last, we need to compute the length of the module $Q$, which equals
the dimension of $M^{\vee\vee}\bigl/M$ as a $\mathbb{C}$-vector
space. The computation is performed by the function \texttt{qLength}
using a presentation matrix for $M$; the actual algorithm is precisely
the one described in \cite[Lemma 2.1 (iii)]{gs}.

\section{Computation of \texorpdfstring{$h^0\bigl(X_k; R^1{\pi_{*}}E\bigr)$}{h\^{}0(R\^{}1pi\_*E)}}\label{sec.algheight}

Let $E$ be a bundle on $Z_k$ of type \eqref{eq.Eext} determined by
$(j,p)$. The sheaf $R^1{\pi_{*}}E$ is supported at the origin, since
$\pi$ is an isomorphism everywhere else. Therefore $H^0\bigl(X_k;
R^1{\pi_{*}}E\bigr) \cong \bigl(R^1{\pi_{*}}E\bigr)_0$. The Theorem of
Formal Function gives
\[ \bigl(R^1{\pi_{*}}E\bigr)_0^\wedge = \varprojlim_n
   H^1\bigl(\ell^{(n)}; E^{(n)}\bigr) \text{ .} \]
However, this limit stabilises at a finite $n$, and so we may simply
compute the finite-dimensional vector space $H^1\bigl(Z_k; E\bigr)$;
then its dimension is the height of $E$.

In this section we present an algorithm that produces a basis for
$H^1\bigl(Z_k; E\bigr)$. For this we use the \v{C}ech description
\[ H^1\bigl(Z_k; E\bigr) = \frac{\Gamma(E; U \cap V)}{\Gamma(E;U) \oplus \Gamma(E;V)} \text{ ,} \]
so we are looking for sections of $E$ on the overlap $U \cap V$ modulo
sections on either $U$ or $V$. We recall that $U$ and $V$ are affine,
and we may consider our sheaves either as analytic sheaves over
complex spaces or as sheaves over algebraic schemes; both points of
view give the same results.

\begin{proposition}[{\cite[Lemma 2.9]{bgk1}}]
Let $E$ be determined by $(j,p)$. Then every $1$-cocycle in
$H^1\bigl(Z_k; E\bigr)$ can be represented locally over $U$ as
\begin{equation}\label{eq.cancyc}
  \sum_{r=0}^{\bigl\lfloor\small\frac{j-2}{k}\bigr\rfloor}
  \sum_{s=kr-j+1}^{-1} \begin{pmatrix}a_{rs}\\0\end{pmatrix}
  z^s u^r \text{ .}
\end{equation}
\end{proposition}

The idea is the following: The vector space $H^1\bigl(Z_k; E\bigr)$ is
certainly spanned by all the monomial cocycles $c_{rs} \ce (a_{rs}z^s
u^r,0)^T$ from Equation \eqref{eq.cancyc}, so we need to identify which
linear combinations of the $c_{rs}$ vanish in cohomology. But $c_{rs}$
vanishes in cohomology precisely if there is a function $b$
holomorphic on $U$ such that
\begin{equation}\label{eq.h1check}
  T\begin{pmatrix} a_{rs}z^s u^r \\ b \end{pmatrix} = \begin{pmatrix}
 a_{rs} z^{s+j} u^r + p b \\ z^{-j} b \end{pmatrix}
\end{equation}
is holomorphic on $V$. (Here $T$ is the transition matrix for $E$ from
Equation \eqref{eq.transf}.) Since $p$ is a polynomial, only finitely
many terms in $b$ need to be considered, and we obtain an algorithm.

First note that if $p=0$, then there can be no relations among the
$c_{rs}$, and $H^1\bigl(Z_k; E\bigr) = \langle \{ c_{rs} \}
\rangle_\mathbb{C}$.

\begin{proposition}\label{prop.h1b}
If $p \neq 0$, let $\minU$ be the smallest degree of $u$ appearing in
$p$. To obtain $H^1\bigl(Z_k; E\bigr)$, it suffices to check Equation
\eqref{eq.h1check} for polynomials of the form
\[ b(z,u) = \sum_{r=0}^{\bigl\lfloor \small\frac{j-2}{k} \bigr\rfloor
   - \minU} \sum_{s=-j}^{kr} b_{rs} z^s u^r \text{ .} \]
\end{proposition}
\begin{proof}
This is immediate from the form of $p$ in Proposition \ref{prop.canp}
and Equation \eqref{eq.h1check}.
\end{proof}

\subsection{Description of the algorithm}

The algorithm itself consists of two parts: The first part computes
all the linear relations between the generators $c_{rs}$; it returns a
list of all basis elements for $H^1\bigl(Z_k; E\bigr)$ and a set of
relations, which may contain lots of redundant information.

The second part of the algorithm takes these sets of generators and
relations and reduces them to a minimal set of generators and
relations.  From this new data, we compute the dimension of
$H^1\bigl(Z_k; E\bigr)$ as the minimal number of generators minus the
minimal number of relations.

\paragraph{First part: finding relations}
\begin{enumerate}
\item Let $b$ be as in Proposition \ref{prop.h1b}, treating all the
      coefficients as indeterminates.
\item For each monomial $z^s u^r$ for $(r,s)$ in $\bigl\{ (r,s) : r = 0,
      \dotsc, \left\lfloor\frac{j-2}{k}\right\rfloor \text{ and } s =
      kr - j + 1, \dotsc, -1 \bigr\}$, do the following:
      \begin{enumerate}
      \item Let $S$ be the set of all terms in $pb$ with degree
            $(r,s)$ in $(u,z)$.
      \item If $S = \varnothing$, then $c_{rs} \ce (a_{rs}z^su^r,0)^T$
            is an independent generator of $H^1\bigl(Z_k; E\bigr)$.
      \item Otherwise, if $S$ is non-empty, let $B$ be the set of all
            coefficients $b_{il}$ appearing in $S$, and let
            $b'=\sum_{b_{il}\in B} b_{il}z^lu^i$. Note that $z^s u^r$ is
            proportional to at least one term of $pb'$ by construction.
      \item Remove from $pb'$ all terms that are proportional to $z^s u^r$
            and all terms which are holomorphic on $U$; call the result $q$.
      \item Finally, let $Q$ be the set of terms in $q$ that is
            \emph{not} holomorphic on $V$. If $Q=\varnothing$, then
            the cycle $c_{rs}$ vanishes in cohomology, otherwise we
            keep $c_{rs}$ as a non-trivial generator and obtain the
            relation $z^s u^r + \sum_{t\in Q}t = 0$.
      \end{enumerate}
\end{enumerate}

\paragraph{Implementation.} Name: \texttt{iHeight}.
Input: \texttt{(k,p,j)} corresponding to the bundle $E$ determined by
$(j,p)$ on $Z_k$. Output: a pair \texttt{(G,R)}, where \texttt{G} is a
set of monomials $t$ such that the cycles $(t,0)^T$ span a vector
space $V$, and \texttt{R} is a set of linear relations on $V$
(involving the coefficients $b_{rs}$) such that $H^1\bigl(Z_k; E\bigr)
= V\bigl/R$.

\noindent\emph{Pseudo code.}\\
\begin{raggedright}
\verb!  minU   := !minimal $u$-degree occurring in \texttt{p}\\
\verb!  if p = 0 then minU = 0!\\
\verb!  bMax   := floor((j-2)/k) - minU!\\
\qquad make all indices \texttt{b\_(r,s)} for $\mathrm{r} = 0, \dotsc,
\mathtt{bMax}$ and $\mathrm{s}=-\mathrm{j}, \dotsc, \mathrm{k*r}$\\
\verb!  bPoly  := !$\sum_{\mathrm{r},\mathrm{s}}$ \texttt{b\_(r,s) * u\^{}r * z\^{}s}\\
\verb!  pb     := bPoly * p!\\
\verb!  aList  := ! list of terms \texttt{u\^{}r*z\^{}s} for
$\mathtt{r}=0,\dotsc,\mathtt{floor((j-2)/k)}$ and $\mathtt{s}=\mathtt{k*r-j}+1,\dotsc,-1$\\
\verb!  aNonTrivials := {}            //! These two variables\\
\verb!  aRelations   := {}            //! store the final result\\[\baselineskip]
\verb!  for each aCycle in aList do!\\
\verb!    //! To begin, find all terms in \texttt{p*b} that cancel \texttt{aCycle}\\
\verb!    pbpruned := ! all terms from \texttt{pb} with the same $(z,u)$-degree as \texttt{aCycle}\\[\baselineskip]
\verb!    if (pbpruned = {}) then     //! nothing can cancel \texttt{aCycle}\\
\verb!      aNonTrivials := aNonTrivials + {aCycle}!\\
\verb!      continue!\\
\verb!    else!\\
\verb!      leftb  := ! all terms in \texttt{bPoly} that contain coefficients in \texttt{pbpruned}\\
\verb!      leftpb  := leftb * p!\\
\verb!      aCycleR := (!all terms in \texttt{leftpb} that are proportional to \verb!aCycle)!\\
\verb!      leftpb  := leftpb - aCycleR!\\
\verb!      leftpb  := leftpb - (!all terms holomorphic in \verb!(z,u))!\\[\baselineskip]
\verb!      leftovers := !those terms of \texttt{z\^{}j * leftpb} that are not holomorphic in $(z^{-1},z^ku)$\\[\baselineskip]
\verb!      if leftovers = 0 then!\\
\verb!        aNonTrivials := aNonTrivials + {aCycle}!\\
\verb!        aRelations   := aRelations + {aCycleR + leftovers}!\\
\verb!      end if!\\[\baselineskip]
\verb!    end if!\\[\baselineskip]
\verb!  end for!\\[\baselineskip]
\verb!  return (aNonTrivials, aRelations)!\\
\end{raggedright}

\paragraph{Second part: reducing to minimal generators and relations}
The first part of the algorithm produces two sets of data: a set $G$ of
generating monomials of for form $z^s u^r$ (i.e.\ the cocycle $(z^s
u^r,0)^T$ is non-trivial in $H^1\bigl(Z_k; E\bigr)$, and a set $R$ of
relations which are polynomials in $z,z^{-1},u$ with coefficients
$b_{ij}$. Let $C$ be the set of all coefficients $b_{ij}$ that can
appear; $C$ is determined by Proposition \ref{prop.h1b}. To find
minimal generators and relations, proceed as follows:
\begin{itemize}
\item Build a new set $R''$ of relations without indeterminates as
      follows: For each relation $r \in R$, for each $\beta \in C$,
      set $\beta=1$ and all other coefficients in $C \setminus
      \{\beta\}$ to zero; add the relation $r\rvert_{\beta = 1, C
      \setminus \{\beta\} = 0}$ to $R''$.
\item Build a new set of generators $G'$ and a new set of relations
      $R'$ by starting with $G'=G$ and $R'=R''$ as follows: Let $N$
      be the set of monomial relations in $R'$, i.e.\ relations of
      the form $z^s u^r=0$. For each $r \in N$, remove $r$ from $G'$
      and substitute $r=0$ into every relation in $R'$. Let $N$ be
      the new set of monomial relations in $R'$ and repeat until
      $N=\varnothing$.
\item The final set $G'$ is a minimal set of generators, and the final
      set $R'$ is a minimal set of relations.
\end{itemize}

\paragraph{Implementation in pseudo code.} Name: \texttt{fixHeightRelations}.
Input: \texttt{(G,R)}, the sets of generators and relations which the
\texttt{iHeight} algorithm produced. Output: a new pair
\texttt{(G',R')}, where \texttt{G'} is a minimal set of generators for
the vector space $H^1\bigl(Z_k; E\bigr)$, and \texttt{R'} is a new set
of linear relations, usually empty. Thus $|\mathtt{G'}|-|\mathtt{R'}|$
is the actual value of the height of $E$ (this number is also returned
in the actual implementation).\\\emph{Pseudo code.}\\
\begin{raggedright}
\verb!  if (G = {} or R = {}) then!\\
\verb!    return (G, R)!\\
\verb!  end if!\\[\baselineskip]
\verb!  rels    := {}    //! this stores the result \texttt{R}\\
\verb!  allvars := !the set of coefficients \texttt{b\_(r,s)}\\[\baselineskip]
\verb!  for each !term \verb!t in G do!\\
\verb!    for each v in allvars do!\\
\verb!      l1 := t !with \texttt{v=1} and all other variables set to zero\\
\verb[      if l1 != 0 then rels = rels + {l1}[\\
\verb!    end for!\\
\verb!  end for!\\[\baselineskip]
\verb!  prunednontrivs := G!\\
\verb!  prunedrels := rels!\\
\verb!  nullguys   := !the set of one-term relations (e.g.\ $z^3u^2=0$) in \texttt{prunedrels}\\[\baselineskip]
\verb[  while (nullguys != {}) do[\\
\verb!    for each !term t\verb! in nullguys do!\\
\verb!      !replace $(R+\mathtt{t})$ by $(R)$ in \verb!prunedrels!\\
\verb!      !replace $(R+\mathtt{t})$ by $(R)$ in \verb!prunednontrivs!\\
\verb!    end for!\\
\verb!    nullguys := !the set of one-term relations in \texttt{prunedrels}\\[\baselineskip]
\verb!  end while!\\[\baselineskip]
\verb!  return (prunednontrivs, prunedrels)!\\
\end{raggedright}

\section{Note on computing \texorpdfstring{$H^1\bigl(Z_k; \SEnd E\bigr)$}{H\^{}1(Z\_k; End E)}}\label{sec.algend1}

In the next two sections we compute invariants of the endomorphism bundle
$\SEnd(E) = E \otimes E^\wedge$. This bundle plays a fundamental role
in the deformation theory of the sheaf $E$: $H^1\bigl(Z_k; \SEnd
E\bigr)$ is precisely the tangent space at $E$ of the moduli of
holomorphic bundles diffeomorphic to $E$. This follows for example
from \cite[Proposition 6.4.3]{donkron}, as the Kuranishi map vanishes
on $Z_k$ (since $H^i\bigl(Z_k;\mathcal{F}\bigr) = 0$ for all $i\geq2$
and every coherent sheaf $\mathcal{F}$).

If $(j,p)$ determines a bundle $E$ on $Z_k$ as before, with transition
function $T$ given by Equation \eqref{eq.transf}, then the endomorphism
bundle $\SEnd E = E \otimes E^\vee$ is a rank-$4$ bundle whose transition function is given, after a convenient change of coordinates
\[ P \ce \begin{pmatrix} 0&1&0&0 \\ 1&0&0&0 \\ 0&0&0&1 \\ 0&0&1&0
   \end{pmatrix} = P^{-1} \text{ ,} \]
by
\[ S \ce P \bigl(T \otimes T^T\bigr) P = \begin{pmatrix}
   1 & z^j p & z^{-j} p & p^2 \\ 0 & z^{2j} & 0 & z^j p \\
   0 & 0 & z^{-2j} & z^{-j} p \\ 0 & 0 & 0 & 1 \end{pmatrix}
   \eqand[, \ \ so] S^{-1} = \mypm{1&-z^{-j}p&-z^jp&p^2}{0&z^{-2j}&
   0&-z^{-j}p}{0&0&z^{2j}&-z^jp}{0&0&0&1}
 \text{ .} \]
The computation of $H^1\bigl(Z_k; \SEnd E\bigr)$ is essentially the
same as for $H^1\bigl(Z_k; E\bigr)$, in the sense that there exists a
simple canonical representative for every $1$-cocycle, and the first
part of the algorithm from \S\ref{sec.algheight} translates almost
literally, while the second part remains unchanged. An implementation
of this algorithm, named \texttt{h1end(k,p,j)}, is contained in our
\emph{Macaulay~2} code at \url{http://www.maths.ed.ac.uk/~s0571100/Instanton/}.

\section{Cancelling infinities: the computation of \texorpdfstring{$H^0\bigl(Z_k;
\SEnd E\bigr)$}{H\^{}0(Z\_k; End E)}}\label{sec.algend0}

On the other hand, the space $H^0\bigl(Z_k; \SEnd E\bigr)$ is
infinite-dimensional, so we cannot compute it directly. However, if we
are given two different bundles $(j,p_1)$ and $(j,p_2)$, we can
compute a ``relative dimension'' of $H^0$-spaces as follows: For each
$n\geq0$, the space $H^0(\ell^{n}; E^{(n)}\bigr)$ is
finite-dimensional, so we can compute the difference
$\Delta_n(p_1,p_2) \ce h^0\bigl(\ell^{(n)}; E(p_1)^{(n)}\bigr) -
h^0\bigl(\ell^{(n)}; E(p_2)^{(n)}\bigr)$. Since $p_1$, $p_2$ are
polynomials, $\Delta_n$ is constant for $n \gg 0$. Finally, we can
define a function
\[ \Delta\bigl(E\bigr) \equiv h\bigl(j,p\bigr) \ce \lim_{n\to\infty} \Delta_n\bigl(0,p\bigr) \text{ ,} \]
which is non-negative since the split bundle given by $p=0$ has the
largest amount of sections on each infinitesimal neighbourhood
$\ell^{(n)}$. (See \cite{bgk1} for a discussion of this non-trivial fact.)

In the remainder of this section we describe an agorithm to compute
$h^0\bigl(\ell^{(n)}; \SEnd E^{(n)}\bigr)$ from the input data $(j,p)$
and $n$. This amounts to finding the most general section $\sigma \in
\Gamma\bigl(\SEnd E; U\bigr)$ such that $S\sigma \in \Gamma\bigr(\SEnd
E; V\bigr)$. Let $\sigma = (a,b,c,d)^T$, where we write a typical
component as $a(z,u) = \sum_{r,s\geq0} a_{rs} z^s u^r$. We have
\begin{equation}\label{eq.Sabcd}
  S\mypm abcd = \mypm{a + z^j p b + z^{-j} p c + p^2 d}{z^{2j} b
  + z^j p d}{z^{-2j} c + z^{-j} p d}d \text{ .}
\end{equation}

For computing $H^0$, we see that $d$ and we know that $z^{-j}p$ is holomorphic on $V$, thus so is $z^{-j}pd$. This implies immediately that
\[ d(z,u) = \sum_{r\geq0} \sum_{s=0}^{kr} d_{rs} z^s u^r \eqand
   c(z,u) = \sum_{r\geq0} \sum_{s=0}^{kr+2j} c_{rs} z^s u^r \text{ .}
\] Next, in the second entry of \eqref{eq.Sabcd}, $z^j p d$ contains
powers $z^s u^r$ for $s \leq kr+2j-1$, so that we may take $b(z,u) =
\sum_{r\geq0} \sum_{s=0}^{kr-1}$ (we could have taken into account the
fact that $u$ divides $p$ for an even sharper bound, but choose not
to). Finally, similar considerations show that we may take $a(z,u) =
\sum_{r\geq0} \sum_{s=0}^{kr+2j-1} a_{rs} z^s u^r$. Here we assume
that $j>0$ and that $p$ is in canonical form \eqref{eq.pnormal}; if
$j=0$, there exists only one bundle anyway.

We now describe an algorithm that computes for a given $n\geq0$ the
vector space $H^0\bigl(\ell^{(n)}; E^{(n)}\bigr)$. The algorithm
generates a set of linear relations on a larger vector space spanned
by monomials, and in our reference implementation we use a built-in
function from \emph{Macaulay~2} to compute the dimension of the
resulting space directly.

\paragraph{Implementation.} Name: \texttt{h0end}. Input: \texttt{(k,p,j,n)},
where \texttt{(k,p,j)} determine as before a bundle $E$ on $Z_k$, and
\texttt{n} is the infinitesimal neighbourhood. Output:
$h^0\bigl(\ell^{(n)}; E^{(n)}\bigr)$.\\\emph{Pseudo code.}\\
\begin{raggedright}
\verb!  aPoly := !$\sum_{r=0}^n \sum_{s=0}^{kr+2j}$ \texttt{a\_(r,s) z\^{}s u\^{}r}\\[1ex]
\verb!  bPoly := !$\sum_{r=0}^n \sum_{s=0}^{kr}$ \texttt{b\_(r,s) z\^{}s u\^{}r}\\[1ex]
\verb!  cPoly := !$\sum_{r=0}^n \sum_{s=0}^{kr+2j}$ \texttt{c\_(r,s) z\^{}s u\^{}r}\\[1ex]
\verb!  dPoly := !$\sum_{r=0}^n \sum_{s=0}^{kr}$ \texttt{d\_(r,s) z\^{}s u\^{}r}\\[1ex]
\verb!  !\\
\verb!  e1 := aPoly + z^j * p * bPoly + z^(-j) * p * cPoly + p^2 * dPoly!\\
\verb!  e2 := z^(2*j) * bPoly + z^j * p * dPoly!\\
\verb!  e3 := z^(-2*j) * cPoly + z^(-j) * p * dPoly!\\
\verb!  e4 := dPoly!\\
\verb!  !\\
\verb!  relations := {}!\\
\verb!  for each !polynomial \verb!pol in (e1, e2, e3, e4} do!\\
\verb!    s := z-degree(t)!\\
\verb!    r := u-degree(t)!\\
\verb!    for each !term \verb!t in pol do!\\
\verb!      if s <= k * r then!\\
\verb!        continue!\\
\verb!      end if!\\[\baselineskip]
\verb!      badterms  := !terms in \texttt{pol} of the same degree as \texttt{t}\\
\verb!      badcoefs  := badterms / z^s * u^r!\\
\verb!      relations := relations + {badcoefs = 0}!\\
\verb!  end for!\\
\verb!  !\\
\verb!  return! (the dimension of $\langle a_{rs},b_{rs},c_{rs},d_{rs}\rangle\bigl/$\texttt{relations})\\
\end{raggedright}

\section{Note on adapting the algorithms to other spaces}\label{sec.adapt}

To conclude this paper we outline how to adapt the algorithms from
this paper to any space $X = \Tot\bigl(\bigoplus_{i=1}^n
\op(a_i)\bigr)$ for which we have a GAGA principle and in which we can
contract the zero section. We refer to \cite{bgk1} and infer that $X$
satisfies GAGA at least when we have $a_i<0$ for all $i$, in which
case all vector bundles are algebraically filtered, and in particular
every rank-$2$ bundle is an algebraic extension of algebraic line
bundles.

The changes in the algorithm are basically as follows: We now need coordinate charts
\[ U = \bigl\{ (z, \oton{u}) \eqand V = \bigl\{(w, \oton{v})\bigr\} \]
which patch together via $w=z^{-1}$ and $v_i = z^{-a_i}u_i$. If $E$ is
a bundle on $X$ of type \eqref{eq.Eext}, then its transition function
is determined by the splitting type $j$ and a polynomial $p \in
\mathbb{C}[z^{\pm1},\oton{u}]$, which can be put into a canonical form.

\paragraph{Example.} On the Calabi-Yau threefold $W_1 \ce \Tot\bigl(\op(-1)
\oplus \op(-1)\bigr)$, the polynomial $p$ has the canonical form
\[ p(z,u,v) = \sum_{t=\epsilon}^{2j-2} \ \ \sum_{r=1-\epsilon}^{2j-2-t} \ \ \sum_{s=r+t-j+1}^{j-1} \ p_{trs} \; z^s u_1^r u_2^t \text{ , \ } \epsilon\in\{1,2\} \text{ .} \]

The putative local sections of $E$ and $\SEnd E$ that one uses in the
computation of $l(Q)$ and $h^i\bigl(X;\SEnd E\bigr)$ are now also
polynomials in $z, \oton{u}$, but otherwise the algorithms are
essentially identical. The adapted algorithms for the flop space $W_1$
are available together with the previous algorithms on the author's
website.

We wish to state a last result, on which the author hit only
after adapting the algorithm to $W_1$ and always computing zero.
\begin{proposition}[cf.\ \cite{gk1}]
Let $X$ be as above, with contractible zero section isomorphic to
$\mathbb{P}^1$, and let $\pi \colon X \to X'$ be the contraction. If
$\dim X > 2$, then for any locally free sheaf $\mathcal{E}$ on $X$,
the width of $\mathcal{E}$ vanishes, i.e.\
\[ H^0\bigl(X'; (\pi_*\mathcal{E})^{\vee\vee}\bigl/(\pi_*\mathcal{E}) \bigr) = 0 \text{ .} \]
\end{proposition}

\appendix

\section{Application: Rank-\texorpdfstring{$2$}{2} Bundles on \texorpdfstring{$Z_k$}{Z\_k} and genericity}

In this section we tabulate some results that were computed with the
algorithms from this paper. We will consider bundles of rank $2$ on
the surfaces $Z_k \ce \Tot\bigl(\op(-k)\bigr)$ of splitting type $j$
which are determined by a polynomial $p$ according to Equation
\eqref{eq.transf}. The moduli of such bundles has been studied in
\cite{bgk1}, and here we provide concrete examples of elements of the
moduli and their numerical invariants.

This section is divided into two cases, first the \emph{instanton
case} for which $j=nk$ for some positive integer $n$, followed by
non-instanton bundles. This terminology arises because it was shown in
\cite{gkm} and \cite{bgk1} that a holomorphic bundle of the type which
we consider corresponds to an instanton on $Z_k$ precisely when $j=nk$
(via a version of the Kobayashi-Hitchin correspondence).

\subsection{The instanton case}

In the case that $j=nk$ for some $n\in\mathbb{N}$, the bundle $E$
determined by $(j,p)$ corresponds to an instanton on $Z_k$. In this
case it was shown that the split bundle is the unique one with the
highest invariants $(h,w)$ (the height and width), while the generic
bundles are those with the lowest invariants $(h,w)$. Our computations
indicate that the condition of being generic or split also corresponds
to $h^1\bigl(\SEnd E\bigr)$ being respectively minimal or maximal.

\noindent\hfil\begin{longtable}{lllcccccl}\toprule
$k$ & $j$ & $p$ & $h^{\vphantom{\bigl(}1}\bigl(\SEnd E\bigr)$ &
$\Delta$ & $h^1-\Delta$ & $(w,h)$ \\\midrule\endhead
\multicolumn{8}{r}{{\itshape\scriptsize Continued on next page}}\\\bottomrule\endfoot
\bottomrule\endlastfoot%
$1$ & $2$ & $u$, $zu$ & $4$ & $2$ & $2$ & $(1,1)$ \\
$1$ & $2$ & $zu^2$    & $5$ & $1$ & $4$ & $(2,1)$ \\
$1$ & $2$ & $0$       & $6$ & $0$ & $6$ & $(3,1)$ \\
\midrule
$1$ & $3$ & $z^{-1}u$, $z^2u$ & $11$ & $4$ & $7$ & $(3,2)$ \\
$1$ & $3$ & $u$, $zu$         & $9$  & $6$ & $3$ & $(1,2)$ \\
$1$ & $3$ & $z^{-1}u+z^2u$    & $9$ & $6$ & $3$ & $(1,2)$ \\
$1$ & $3$ & $u^2$, $z^2u^2$   & $12$ & $3$ & $9$ & $(3,3)$ \\
$1$ & $3$ & $zu^2$            & $11$ & $4$ & $7$ & $(2,3)$ \\
$1$ & $3$ & $u^2+z^2u^2$      & $11$ & $4$ & $7$ & $(2,3)$ \\
$1$ & $3$ & $zu^3$, $z^2u^3$  & $13$ & $2$ & $11$ & $(4,3)$ \\
$1$ & $3$ & $z^2u^4$          & $14$ & $1$ & $13$ & $(5,3)$ \\
$1$ & $3$ & $0$               & $15$ & $0$ & $15$ & $(6,3)$ \\
\midrule
$2$ & $6$ & $z^{-3}u,z^5u$  & $31$ & $5$ & $26$ & $(6,7)$ \\
$2$ & $6$ & $z^{-2}u,z^4u$  & $28$ & $8$ & $20$ & $(4,6)$ \\
$2$ & $6$ & $z^{-1}u,z^3u$  & $25$ & $11$ & $14$ & $(2,5)$ \\
$2$ & $6$ & $u,z^2u$        & $24$ & $12$ & $12$ & $(1,5)$ \\
$2$ & $6$ & $zu$            & $23$ & $13$ & $10$ & $(0,5)$ \\
$2$ & $6$ & $z^{-3}u+z^2u$  & $23$ & $13$ & $10$ & $(0,5)$ \\
$2$ & $6$ & $z^{-2}u+z^3u$  & $23$ & $13$ & $10$ & $(0,5)$ \\
$2$ & $6$ & $z^{-1}u+z^4u$  & $23$ & $13$ & $10$ & $(0,5)$ \\
$2$ & $6$ & $u+z^2u$        & $24$ & $12$ & $12$ & $(1,5)$ \\
$2$ & $6$ & $z^{-1}u+z^3u$  & $24$ & $12$ & $12$ & $(1,5)$ \\
$2$ & $6$ & $z^{-2}u+z^4u$  & $24$ & $12$ & $12$ & $(1,5)$ \\
$2$ & $6$ & $z^{-1}u^2$, $z^5u^2$ & $32$ & $4$ & $28$ & $(6,8)$ \\
$2$ & $6$ & $u^2$, $z^4u^2$ & $30$ & $6$ & $24$ & $(4,8)$ \\
$2$ & $6$ & $z^{\{1,2,3\}}u^2$    & $28$ & $8$ & $20$ & $(2,8)$ \\
$2$ & $6$ & $u^2+z^4u^2$    & $28$ & $8$ & $20$ & $(2,8)$ \\
$2$ & $6$ & $\ldots u^3$    & $$ & $$ & $$ & $$ \\
$2$ & $6$ & $z^{\{3,4,5\}}u^4$    & $34$ & $2$ & $32$ & $(7,9)$ \\
$2$ & $6$ & $z^5u^5$        & $35$ & $1$ & $34$ & $(8,9)$ \\
$2$ & $6$ & $0$             & $36$ & $0$ & $36$ & $(9,9)$ \\
\midrule
$3$ & $3$ & $zu$         & $6$ & $1$ & $5$ & $(0,2)$ & & \\
$3$ & $3$ & $0$          & $7$ & $0$ & $7$ & $(1,2)$ & & \\
\midrule
$3$ & $6$ & $zu$         & $19$ & $7$ & $12$ & $(0,5)$ & & \\
$3$ & $6$ & $0$          & $26$ & $0$ & $26$ & $(5,7)$ & & \\
\midrule
$3$ & $9$ & $zu$         & $38$ & $19$ & $19$ & $(0,8)$ & & \\
$3$ & $9$ & $0$          & $57$ & $0$ & $57$ & $(12,15)$ & & \\
\midrule
$4$ & $4$ & $zu$         & $9$ & $1$ & $8$ & $(0,3)$ & & \\
$4$ & $4$ & $0$          & $10$ & $0$ & $10$ & $(1,3)$ & & \\
\midrule
$4$ & $8$ & $zu$         & $27$ & $$ & $$ & $(0,7)$ & & \\
$4$ & $8$ & $0$          & $36$ & $0$ & $36$ & $(6,10)$ & & \\
\midrule
$4$ & $12$ & $zu$         & $53$ & $$ & $$ & $(0,11)$ & & \\
$4$ & $12$ & $0$          & $78$ & $0$ & $78$ & $(15,21)$ & &
\end{longtable}\hfil

In fact, the numerical evidence leads us to conjecture one and
discover another relation between the invariants $(h,w)$ and
$(h^1,\Delta)$.

\paragraph{Conjecture:} $w+h = \chi = \bigl((h^1-\Delta)-j\bigr)\bigl/2 + j\bigl/k$.

\begin{proposition}\label{prop.h0h1}
  $\Delta + h^1 = h^1\bigl(\SEnd(\text{split})\bigr)$, or equivalently
\[ h^1\bigl(\SEnd E\bigr) - h^0\bigl(\SEnd E\rvert_{\ell^{(m)}}\bigr)
   = h^1\bigl(\SEnd(\text{split})\bigr) - h^0\bigl(\SEnd(\text{split}\rvert_{\ell^{(m)}})\bigr) \]
for all large $m$.
\end{proposition}
\begin{proof}
We can express the statement in terms of the Hilbert polynomial
\[ \phi_{E,m}(n) \ce \chi\bigl(\SEnd E^{(m)}(n)\bigr) = h^1\bigl(
   \SEnd E\rvert_{\ell^{(m)}}(n)\bigr) - h^0\bigl(\SEnd E
   \rvert_{\ell^{(m)}}(n) \bigr) \text{ ;} \]
then the statement is $\phi_{E,m}(0) = \phi_{\mathrm{split},m}(0)$.
But in fact we have $\phi_{E,m}(n) = \phi_{\mathrm{split},m}(n)$ for
any $n$ and $m$, since the Hilbert polynomial cannot distinguish
extensions on $\ell^{(m)} \subset Z_k$, as one sees from direct
computation. We retain the clause ``for all large $m$'' since
$h^1\bigl(\SEnd E\rvert_{\ell^{(m)}}\bigr)$ stabilises to
$h^1\bigl(\SEnd E\bigr)$ eventually.
\end{proof}

From \cite{bgk1} and the Conjecture we get $h^1 + \Delta = n(2nk + k -
2)$ for $j=nk$. We also get that for the generic bundle $p=zu$, we
have $h^1-\Delta=(3k-2)n-2$ for $k>1$ and $h^1-\Delta=n$ for $k=1$.

\paragraph{Corollary:} Assuming the Conjecture, we have
$\Delta = n^2k-\chi$, $h^1=kn(n+1)-2n+\chi$.

\subsection{Non-instanton bundles}

Here $j\not\equiv0 \pmod{k}$.

\noindent\hfil\begin{longtable}{lllcccccl}\toprule
$k$ & $j$ & $p$ & $h^{\vphantom{\bigl(}1}\bigl(\SEnd E\bigr)$ &
$\Delta$ & $h^1-\Delta$ & $(w,h)$ \\\midrule\endhead
\multicolumn{8}{r}{{\itshape\scriptsize Continued on next page}}\\\bottomrule\endfoot
\bottomrule\endlastfoot%
$2$ & $3$ & $u$      & $7$ & $2$ & $5$ & $(1,2)$ \\
$2$ & $3$ & $zu$     & $7$ & $2$ & $5$ & $(0,2)$ \\
$2$ & $3$ & $z^2u$   & $7$ & $2$ & $5$ & $(1,2)$ \\
$2$ & $3$ & $u+z^2u$ & $7$ & $2$ & $5$ & $(0,2)$ \\
$2$ & $3$ & $z^2u^2$ & $8$ & $1$ & $7$ & $(2,2)$ \\
$2$ & $3$ & $0$      & $9$ & $0$ & $9$ & $(2,2)$ \\
\midrule
$3$ & $4$ & $u$       & $10$ & $2$ & $8$ & $(1,3)$ \\
$3$ & $4$ & $zu,z^2u$ & $10$ & $2$ & $8$ & $(0,3)$ \\
$3$ & $4$ & $z^3u$    & $10$ & $2$ & $8$ & $(1,3)$ \\
$3$ & $4$ & $u+z^3u$  & $10$ & $2$ & $8$ & $(0,3)$ \\
$3$ & $4$ & $z^3u^2$  & $11$ & $1$ & $10$ & $(2,3)$ \\
$3$ & $4$ & $0$       & $12$ & $0$ & $12$ & $(2,3)$ \\
\midrule
$3$ & $5$ & $z^{-1}u$ & $16$ & $2$ & $14$ & $(2,4)$ \\
$3$ & $5$ & $u$       & $15$ & $3$ & $12$ & $(1,4)$ \\
$3$ & $5$ & $zu,z^2u$ & $14$ & $4$ & $10$ & $(0,4)$ \\
$3$ & $5$ & $z^3u$    & $15$ & $3$ & $12$ & $(1,4)$ \\
$3$ & $5$ & $z^4u$    & $16$ & $2$ & $14$ & $(2,4)$ \\
$3$ & $5$ & $u+z^4u$  & $14$ & $4$ & $10$ & $(0,4)$ \\
$3$ & $5$ & $z^{-1}u+z^4u$ & $15$ & $3$ & $12$ & $(1,4)$ \\
$3$ & $5$ & $z^2u^2$  & $17$ & $1$ & $16$ & $(2,5)$ \\
$3$ & $5$ & $z^3u^2$  & $17$ & $1$ & $16$ & $(2,5)$ \\
$3$ & $5$ & $z^4u^2$  & $17$ & $1$ & $16$ & $(2,5)$ \\
$3$ & $5$ & $0$       & $18$ & $0$ & $0$ & $(3,5)$
\end{longtable}

We see that the invariants $(w,h)$ are now required to distinguish the
generic bundles (those with lowest $(w,h)$, wheras the split bundle is
no longer the only one with the highest values of $(w,h)$ (see
\cite{bgk1} for details). However, $h^1\bigl(\SEnd E\bigr)$ still
distinguishes the split bundle. The physical interpretation of these
``non-instanton'' bundles invites further exploration.

\section{Usage example}

The algorithms that we described in this paper are implemented in
\emph{Macaulay 2} version 1.1, the code is contained in the file
\texttt{InstantonInvariants2.m2}. Suppose we want to study the bundle
on $Z_2$ of splitting type $7$ given by the polynomial
$p=z^{-1}u+zu^2$. We set \emph{Macaulay} up as follows:

\bigskip\noindent
\begin{raggedright}
\verb!$ M2 InstantonInvariants2.m2!\\
\verb!Macaulay 2, version 1.1!\\
\verb!with packages: ...!\\
\verb!  !\\
\verb!i1 : p = z*u^2+z^-1*u;!\\
\end{raggedright}
\medskip\noindent
Now we compute the width and height, and $h^1\bigl(Z_2;\SEnd E\bigr)$. When calling \texttt{iWidth}, we use the option \texttt{Verbose=>false} to suppress additional output.
\medskip\noindent
\begin{raggedright}
\verb!i2 : iWidth(2,p,7,Verbose=>false)!\\
\verb!i3 : iHeight(2,p,7)!\\
\verb!i4 : fixHeightRelations oo!\\
\verb!i5 : h1End(2,p,7)!\\
\verb!i6 : fixHeightRelations oo!\\
\end{raggedright}
\medskip\noindent
We find that the width is $2$, the height $6$, and $h^1\bigl(Z_2;\SEnd E\bigr) = 33$.

\bigskip\bigskip\vfill

\noindent Thomas K\"{o}ppe \\
School of Mathematics, The University of Edinburgh \\
James Clerk Maxwell Building, The King's Buildings \\
Mayfield Road, Edinburgh, UK, EH9 3JZ \\
E-mail: \url{t.koeppe@ed.ac.uk}

\end{document}